\documentclass{amsart}[12pt]
\usepackage{amsmath, amsthm, amscd, amsfonts,}
\usepackage{graphicx,float}

\usepackage{amssymb}

\usepackage{pb-diagram}
\usepackage{lamsarrow}
\usepackage{pb-lams}

\DeclareMathOperator{\Add}{Add}

\def\k{\kappa}

\def\l{\lambda}

\def\a{\alpha}

\def\b{\beta}

\newtheorem{theorem}{Theorem}[section]
\newtheorem{lemma}[theorem]{Lemma}

\numberwithin{equation}{section}

\usepackage{pb-diagram}

\def\l{\lambda}

\def\rmark{\mbox{$\rm\bf\rule{0.06em}{1.45ex}\kern-0.05em R$}}
\def\pmark{\mbox{$\rm\bf\rule{0.06em}{1.45ex}\kern-0.05em P$}}
\def\nmark{\mbox{$\rm\bf\rule{0.06em}{1.45ex}\kern-0.05em N$}}
\def\vdash{\mbox{$\rm\| \kern-0.13em -$}}

\def\l{\lambda}

\def\rmark{\mbox{$\rm\bf\rule{0.06em}{1.45ex}\kern-0.05em R$}}
\def\pmark{\mbox{$\rm\bf\rule{0.06em}{1.45ex}\kern-0.05em P$}}
\def\nmark{\mbox{$\rm\bf\rule{0.06em}{1.45ex}\kern-0.05em N$}}
\def\vdash{\mbox{$\rm\| \kern-0.13em -$}}
\newcommand{\lusim}[1]{\smash{\underset{\raisebox{1.2pt}[0cm][0cm]{$\sim$}}
{{#1}}}}

\begin{document}

\title[Woodin's surgery method]{Woodin's surgery method}

\author[Mohammad Golshani]{Mohammad
  Golshani}

  \thanks{The author's research was in part supported by a grant from IPM (No. 91030417). He also wishes to thanks Prof. Woodin for letting him to present Theorem 1.1 in the paper.}

\thanks{} \maketitle

%{ \\ Department of Mathematics\\  Shahid Bahonar University of Kerman, Kerman, Iran}

%\subjclass[2010]{03E35}

%\keywords{ Multiplication mudule, prime submodule, Strongly prime submodule, valuation,
%fractional submodule, pseudo-valuation module}
\begin{abstract}
In this short paper we give an overview of Woodin's surgery method.
\end{abstract}

\section{Surgery method for strong cardinals}
In this section we present an abstract version of Woodin's surgery method for strong cardinals.
\begin{theorem}
(\cite{woodin}) Let $j: M \rightarrow N$ be an elementary embedding with $\k=crit(j),$ where $\k$ is inaccessible in $M,$ and $N=\{j(F)(a): F\in M, F:[\k]^{<\omega}  \rightarrow M$ and $a\in [\l]^{<\omega} \}.$ Let $\mathbb{P}=Add(\k,\nu)_M,$ where $\nu$ is a cardinal in $M$ and suppose that $j\upharpoonright\nu\in M.$ Let $G$ be $\mathbb{P}-$generic over $M$, and suppose that there exists $H$ such that:
\begin{enumerate}
\item $N \subseteq M[G],$
\item $M[G]\models N^\k \subseteq N,$
\item $H$ is $j(\mathbb{P})-$generic over $M[G].$
\end{enumerate}
Then there exists $H^*\in M[G][H]$ such that $H^*$ is $j(\mathbb{P})-$generic over $N$ and $j[G] \subseteq H^*.$
\end{theorem}
We will present two proofs of the above theorem. The first one, essentially due to Woodin, is taken from \cite{cummings}.
\begin{center}
\end{center}
{\bf First proof.} Let
\begin{center}
$g=\bigcup G: \nu \rightarrow 2,$

$h= \bigcup H: j(\nu) \rightarrow 2.$
\end{center}
Define $h^*: j(\nu) \rightarrow 2$ by
\begin{center}
 $h^*(\b) = \left\{ \begin{array}{l}
       g(\a)  \hspace{1.1cm} \text{ if } \b=j(\a),\\
       h(\b)  \hspace{1.2cm} \text{}  Otherwise.
 \end{array} \right.$
\end{center}
Let $H^*$ be the filter generated by $h^*.$ Note that $H^*=\{p^*: p\in H \},$ where for each $p\in j(\mathbb{P}),$ $p^*$ is defined by
\begin{itemize}
\item $dom(p^*)=dom(p),$
\item $p^*$ is defined by

\begin{center}
 $p^*(\b) = \left\{ \begin{array}{l}
       g(\a)  \hspace{1.1cm} \text{ if } \b=j(\a),\\
       p(\b)  \hspace{1.2cm} \text{}  Otherwise.
 \end{array} \right.$
 \end{center}
\end{itemize}
Let's first show that $H^*$ is well-defined.
\begin{lemma}
$p\in j(\mathbb{P}) \Rightarrow p^* \in j(\mathbb{P}).$
\end{lemma}
\begin{proof}
It suffices to show that $p^*\in N.$ But clearly $p^*\in M[G],$ so by clause $(2)$ of the theorem, it suffices to show that $X(p, p^*)=\{z: p(z)\neq p^*(z) \}$ has size $\leq \k.$ We have $X(p, p^*) \subseteq dom(p) \cap j[\nu],$ so it suffices to show that the later set has size at most $\k.$ Let $p=j(F)(a),$ where $a\in [\l]^{<\omega}, F\in M, F:[\k]^{|a|} \rightarrow M.$ We may further suppose that $\forall x\in [\k]^{|a|}, F(x)\in \mathbb{P}.$ Then
\begin{center}
$\a<\nu, j(\a)\in dom(p) \Rightarrow \exists x, \a\in dom(F(x)).$
\end{center}
So if $X= \bigcup \{dom(F(x)): x\in [\k]^{|a|} \},$ then $X\in M$ and $M[G]\models$`` $|X|\leq \k$ and $dom(p)\cap j[\nu] \subseteq j[X]$''. The result follows.
\end{proof}
It is easily seen that $H^*$ is a filter on $j(\mathbb{P}).$
\begin{lemma}
$H^*$ is $j(\mathbb{P})-$generic over $N$.
\end{lemma}
\begin{proof}
Let $D\in N$ be dense open in $j(\mathbb{P}).$ Define an equivalence relation on $j(\mathbb{P})$ by
\begin{center}
$p \sim q \Leftrightarrow dom(p)=dom(q)$ and $|\{z: p(z) \neq q(z) \}|\leq \k.$
\end{center}

Let $E=\{q\in j(\mathbb{P}): \forall p, p \sim q \Rightarrow p\in D \}.$
We show that $E$ is dense in $j(\mathbb{P}).$ First we prove the following.
\begin{lemma}
If $p\in j(\mathbb{P}),$ then there is $q\leq p$ such that
\begin{center}
$\forall r, r \sim q \Rightarrow r \cup (q\setminus p)\in D.$
\end{center}
\end{lemma}
\begin{proof}
Let $ \langle X_\a: \a<\mu \rangle, \mu < j(\k)$ be an enumeration of $\{X \subseteq dom(p): |X|\leq \k \}.$ Define by induction a decreasing sequence
 $ \langle  p_\a: \a\leq\mu \rangle$ of conditions as follows:
 \begin{itemize}
\item {\bf $\a=0$:} Let $p_0=p$,
\item {\bf $\a=\b+1$:} Suppose $p_\b$ is defined. Let
 \begin{center}
 $q(z) = \left\{ \begin{array}{l}
       p_\b(z)  \hspace{1.8cm} \text{ if } z\in dom(p_\a)\setminus X_\a,\\
       1-p_\b(z)  \hspace{1.3cm} \text{}  Otherwise.
 \end{array} \right.$
 \end{center}
Since $D$ is dense, we can find $\bar{q}\in D$ such that $\bar{q}\leq q.$ Set $p_\a=p_\b \cup (\bar{q}\setminus q).$
\item {\bf $\a$ is a limit ordinal:} Let $p_\a= \bigcup_{\b<\a} p_\b$.
\end{itemize}
Then $q=p_\mu$ is as required.
\end{proof}
\begin{lemma}
$E$ is dense in $j(\mathbb{P}).$
\end{lemma}
\begin{proof}
Let $p\in j(\mathbb{P}).$ Using the above claim $\k^+-$times, we can produce a decreasing sequence $\langle p_\a: \a<\k^+ \rangle$ of conditions extending $p$ such that for any $\a<\k^+$ if $r \sim p_\a,$ then $r\cup (p_{\a+1}\setminus p_\a)\in D.$ Let $q= \bigcup_{\a<\k^+}p_\a$. Then $q\leq p$ and $q\in E.$ To see this just note that if $r\sim q,$ then for some $\a<\k^+, X(r, q) \subseteq dom(p_\a),$ so $X(r, q)=X(r, p_\a).$
\end{proof}
Let $p\in H\cap E.$ Then $p^* \sim p,$ so $p^*\in H^*\cap D.$ The theorem follows.
\end{proof}
\begin{center}
\end{center}
{\bf Second proof.} Let $H^*$ be as defined above. We show that it is $j(\mathbb{P})-$generic over $N$. Thus let $A\in N$ be a maximal antichain of $j(\mathbb{P}).$ Then $|A| \leq j(\k).$ Set
\begin{center}
$S=\bigcup \{dom(p):p\in A \}.$
\end{center}
then $N\models$`` $S \subseteq j(\nu)$ and $|S| \leq j(\k)$''. Let $S=j(F)(a),$ where $a\in [\l]^{<\omega}, F\in M, F:[\k]^{|a|} \rightarrow M.$ We may further suppose that $M \models $`` For each $x\in dom(F), f(x) \subseteq \nu$ and $|f(x)|\leq \k$''. Set $T= \bigcup \{f(x): x\in [\k^{|a|}\}.$ Then $T\in M$ and $M\models$`` $T \subseteq \nu$ and $|T|\leq \k$''. It is easily seen that
\begin{center}
$M[G]\models$``$S\cap j[\nu] \subseteq j[T]$''.
\end{center}
Hence by clause $(2),$ $S\cap j[\nu]\in N.$ Let $X_0=S\cap j[\nu]$ and $X_1=j(\nu)\setminus X_0.$ Also set $\mathbb{P}_i=\{   p\in j(\mathbb{P}): dom(p) \subseteq X_i\}, i=0,1.$
 Then we have a natural forcing isomorphism
\begin{center}
$\pi: j(\mathbb{P}) \rightarrow \mathbb{P}_0\times \mathbb{P}_1,$
\end{center}
given by
\begin{center}
$\pi(p)=  \langle p\upharpoonright X_0, p\upharpoonright X_1 \rangle.$
\end{center}
Note that $H^*\upharpoonright X_0\in \mathbb{P}_0.$ Set
\begin{center}
$A_1=\{p\upharpoonright X_1: p\in A$ and $p$ is compatible with $h^*\upharpoonright X_0 \}.$
\end{center}
The following lemma can be proved quite easily.
\begin{lemma}
$A_1$ is a maximal antichain in $\mathbb{P}_1.$
\end{lemma}
On the other hand $H_1=\{p\upharpoonright X_1: p\in H \}$ is $\mathbb{P}_1-$generic, so $H_1\cap A_1\neq \emptyset.$ Let $p\in A$ be such that $p$ is compatible with
$h^*\upharpoonright X_0$ and $p\upharpoonright X_1 \in H_1 \cap A_1.$ But then $p\in H^*\cap A,$ and hence $H^*\cap A\neq \emptyset.$ The theorem follows. \hfill$\Box$

\section{Surgery method for supercompact cardinals}
In this section we prove the following theorem, which is an analogue of Theorem 1.1 for supercompact cardinals.
\begin{theorem}
Let $j: M \rightarrow N$ be an elementary embedding with $crit(j)=\k,$ where $\k$ is inaccessible in $M,$ and $N=\{j(F)(j[\l]): F\in M, F:P_\k(\l)  \rightarrow M \}$. Let $\mathbb{P}=Add(\k,\nu)_M,$ where $\nu$ is a cardinal in $M$ and suppose that $j\upharpoonright\nu\in M.$ Let $G$ be $\mathbb{P}-$generic over $M$, and suppose that there exists $H$ such that:
\begin{enumerate}
\item $N \subseteq M[G],$
\item $M[G]\models N^\l \subseteq N,$
\item $H$ is $j(\mathbb{P})-$generic over $M[G].$
\end{enumerate}
Then there exists $H^*\in M[G][H]$ such that $H^*$ is $j(\mathbb{P})-$generic over $N$ and $j[G] \subseteq H^*.$
\end{theorem}
\begin{proof}
Let $g, H$ and $H^*$ be defined as before. We show that $H^*$ is as required.

Let $A\in N$ be a maximal antichain of $j(\mathbb{P}).$ Then $|A| \leq j(\k).$ Set
\begin{center}
$S=\bigcup \{dom(p):p\in A \}.$
\end{center}
Then $N\models$`` $S \subseteq j(\nu)$ and $|S| \leq j(\k)$''. Let $S=j(F)(j[\l]),$ where $F\in M, F:P_\k(\l) \rightarrow M.$ We may further suppose that $M \models $`` For each $x\in dom(F), f(x) \subseteq \nu$ and $|f(x)|\leq \k$''. Set $T= \bigcup \{f(x): x\in P_\k(\l)\}.$ Then $T\in M$ and $M\models$`` $T \subseteq \nu$ and $|T|\leq \l$''.
It is easily seen that
\begin{center}
$M[G]\models$``$S\cap j[\nu] \subseteq j[T]$''.
\end{center}
Thus  by clause $(2),$ $S\cap j[\nu]\in N.$ The rest of the argument is as before.
\end{proof}
As an application of the above theorem, we give a proof of the following theorem (compare with \cite{cummings2}, Section 13).
\begin{theorem}

Assume $GCH$ holds and $\k$ is $\k^+-$supercompact. Then there is a generic extension in which $\k$ remains $\k^+-$supercompact and $2^\k=\k^{++}.$
\end{theorem}
\begin{proof}
Let $\mathbb{P}=\mathbb{P}_\k*\lusim{\Add}(\k, \k^{++})$ be the reverse Easton iteration for adding $\a^{++}-$many new Cohen subsets of $\a,$ using $\Add(\a, \a^{++}),$ for all inaccessible cardinals $\a\leq \k,$ and let $G*g$ be $\mathbb{P}-$generic over $V$. Let $j:V  \rightarrow M\simeq Ult(V,U),$ where $U$ is a normal measure on $P_\k(\k^+),$ so that $M=\{j(F)(j[\k^+]): F\in V, F: P_\k(\k^+) \rightarrow V \}.$
Also let $j(\mathbb{P})=\mathbb{P}*\lusim{\mathbb{R}}*\lusim{\Add}(j(\k), j(\k^{++})).$

By standard forcing arguments we can find $H\in V[G*g]$ which is $j(\mathbb{P}_\k)=\mathbb{P}*\lusim{\mathbb{R}}-$generic over $M,$ and since $j[G] \subseteq G*g*H$, We can lift $j$ to some $j: V[G] \rightarrow M[G*g*H].$

Let $h$ be $\Add(j(\k), j(\k^{++}))_{M[G*g*H]}-$generic over $V[G*g*H].$  Applying Theorem 2.1, there exists $h^*$ such that we have the lifting $j^*: V[G*g] \rightarrow M[G*g*H*h^*].$ Working in $V[G*g*H*h],$ define $U^*$ on $P_\k(\k^+)$ by
\begin{center}
$X\in U^* \Leftrightarrow j[\k^+]\in j^*(X).$
\end{center}
Note that
\\
$(*)$$\hspace{3.5cm}$ $V^{\mathbb{P}}\models$``$\mathbb{R}*\lusim{\Add}(j(\k), j(\k^{++}))$ is $\leq \k^+-$closed''.

Using $(*)$, $U^*\in V[G*g],$ and $V[G*g]\models$``$U^*$ is a normal measure on $P_\k(\k^+)$''. The theorem follows.
\end{proof}

School of Mathematics, Institute for Research in Fundamental Sciences (IPM), P.O. Box:
19395-5746, Tehran-Iran.

E-mail address: golshani.m@gmail.com

\end{document}